\documentclass[a4paper,12pt]{amsart}
\usepackage{enumerate}
\usepackage{amsmath,amssymb,amsthm,xspace,amscd,latexsym,amscd}
\usepackage{xy}
\xyoption{all}
\usepackage[T1]{fontenc}

\newtheorem{theorem}{Theorem}
\newtheorem{lemma}[theorem]{Lemma}
\newtheorem{proposition}[theorem]{Proposition}

\newtheorem{remark}[theorem]{Remark}

\begin{document}

\title{Regular strongly typical blocks of $\mathcal{O}^{\mathfrak{q}}$}
\author{Anders Frisk and Volodymyr Mazorchuk}
\date{\today}
\begin{abstract}
We use the technique of Harish-Chandra bimodules to prove that 
regular strongly typical blocks of the category $\mathcal{O}$ for
the queer Lie superalgebra $\mathfrak{q}_n$ are equivalent to the
corresponding blocks of the category  $\mathcal{O}$ for the Lie
algebra $\mathfrak{gl}_n$.
\end{abstract}

\maketitle

\section{Introduction and the main result}\label{s1} 

For $n\in\mathbb{N}$ let $\mathfrak{q}_n$ denote the queer Lie
superalgebra and $\mathcal{O}^{\mathfrak{q}}$ denote the category 
$\mathcal{O}$ for $\mathfrak{q}_n$. The category 
$\mathcal{O}^{\mathfrak{q}}$ decomposes into a direct sum of blocks,
which can be {\em typical} or {\em atypical}. Atypical blocks are
very complicated and may contain infinitely many simple objects.
Typical blocks are much easier and are always equivalent to module
categories over finite-dimensional algebras. In \cite{Fr} it was shown
that the finite-dimensional algebras describing typical blocks are
always properly stratified in the sense of \cite{Dl}. Among all
typical blocks one separates {\em strongly typical} ones, which are 
described (\cite{Fr}) by quasi-hereditary algebras in the sense of
\cite{CPS,DR}. 

In \cite[4.8]{Br} it was conjectured that multiplicities
of simple highest weight supermodules in Verma supermodules for 
$\mathfrak{q}_n$ are given by Kazhdan-Lusztig combinatorics.
In \cite[3.9]{Fr} a much stronger conjecture was formulated, namely
that strongly typical blocks of $\mathcal{O}^{\mathfrak{q}}$ are
equivalent to the corresponding blocks of the category  
$\mathcal{O}$  for the Lie algebra $\mathfrak{gl}_n$. 
A strong evidence for this conjecture, established in \cite{Fr}, 
was a similarity between the quasi-hereditary structures in both
cases. Moreover,  \cite[3.9]{Fr} contains
also an explicit conjecture for the structure of all typical blocks.
The aim of this paper is to prove the conjecture from \cite[3.9]{Fr}
(and hence the conjecture from \cite[4.8]{Br} as well) for 
{\em regular} strongly typical blocks. 

There is a natural restriction functor from $\mathcal{O}^{\mathfrak{q}}$
to the category  $\mathcal{O}$  for the Lie algebra $\mathfrak{gl}_n$.
However, unlike the case of most of the other Lie superalgebras,
this restriction functor does not induce an equivalence in a 
straightforward way. The problem is that the highest weights of 
Verma supermodules over $\mathcal{O}^{\mathfrak{q}}$ are not 
one-dimensional (because the Cartan subsuperalgebra of $\mathfrak{q}$ is
not commutative). Subsequently, under restriction Verma supermodules
are not mapped to the corresponding Verma modules but rather to 
direct sums of Verma modules (see Proposition~\ref{prop2}). 
This suggests that the naive restriction functor is a direct sum of 
several copies of some ``smaller'' functor, which defines the desired 
equivalence of categories. This is exactly what we prove in this paper.

The main idea of the proof is to realize the induction functor
(the left adjoint to the restriction) as a tensor product with 
some Harish-Chandra bimodule.  This requires several definitions and 
some technical  work as we are forced to go beyond the original 
categories and instead work with the so-called thick version 
of the category $\mathcal{O}$. Unfortunately, along the way we use some
properties of Harish-Chandra bimodules, which require an additional
assumption of regularity of the blocks we work with. The main 
result of the paper is the equivalence of blocks of categories
$\mathcal{O}$  for $\mathfrak{q}_n$ and $\mathfrak{gl}_n$,
see Theorem~\ref{thmmain}. This extends earlier results of 
Penkov, Serganova and Gorelik (see \cite{PS,Pe,Go1}) to the case of the
algebra $\mathfrak{q}_n$ and verifies conjectures from \cite[3.9]{Fr}
and \cite[4.8]{Br} in our setup.

The paper is organized as follows: In Section~\ref{s2} we give all
necessary definitions and formulate our main result. In Section~\ref{s3} 
we collect auxiliary technical results about Harish-Chandra bimodules. 
The main result is proved in Section~\ref{s4}. We conclude the paper
with some remarks in Section~\ref{s5}.
\vspace{5mm}

\noindent
{\bf Acknowledgments.} For the second author the research was 
partially supported by the Swedish Research Council.

\section{Definitions, preliminaries and 
formulation of the main result}\label{s2} 

For all undefined notions we refere the reader to \cite{Fr}.
Let $\mathbb{N}$ and $\mathbb{N}_0$ denote the set of all positive 
and nonnegative integers, respectively, and fix $n\in \mathbb{N}$. 
Set $\mathbb{N}_n=\{1,\dots,n\}$. For an
algebraically closed field $\Bbbk$ of characteristic zero let
$\mathtt{g}=\mathfrak{gl}_n$ denote  the general linear Lie algebra 
of $n\times n$  matrices over $\Bbbk$. Let $\mathfrak{q}=\mathfrak{q}_n$ 
denote  the {\em queer Lie superalgebra} over $\Bbbk$, which
consists of block matrices  of the form 
\begin{displaymath}
\mathtt{M}(A,B)=\left(\begin{array}{cc}A&B\\B&A\end{array}\right),\quad
A,B\in \mathfrak{gl}_n.
\end{displaymath}
The even and the odd spaces $\mathfrak{q}_{\overline{0}}$ and
$\mathfrak{q}_{\overline{1}}$ consist of the  matrices $\mathtt{M}(A,0)$
and $\mathtt{M}(0,B)$, respectively, and we have 
$\mathfrak{q}=\mathfrak{q}_{\overline{0}}\oplus\mathfrak{q}_{\overline{1}}$.
For a homogeneous element $X\in \mathfrak{q}$ we denote by 
$\overline{X}\in\mathbb{Z}/2\mathbb{Z}$ the degree of $X$. 
Then the Lie superbracket in $\mathfrak{q}$ is given by  
$[X,Y]=XY-(-1)^{\overline{X}\overline{Y}}YX$, where $X,Y\in \mathfrak{q}$ 
are homogeneous.

For $i,j\in\mathbb{N}_n$ let $E_{ij}\in \mathfrak{gl}_n$ denote the
corresponding matrix unit. We have the {\em Cartan subsuperalgebra} 
$\mathfrak{h}$ of $\mathfrak{q}$, which is the linear span of 
$H_i=\mathtt{M}(E_{ii},0)$ and $\mathtt{M}(0,E_{ii})$, $i\in\mathbb{N}_n$. 
The superalgebra $\mathfrak{h}$ inherits from  $\mathfrak{q}$ the 
decomposition $\mathfrak{h}= 
\mathfrak{h}_{\overline{0}}\oplus\mathfrak{h}_{\overline{1}}$.

Let $\{\varepsilon_i:i\in\mathbb{N}_n\}$, denote the basis of 
$\mathfrak{h}^*_{\overline{0}}$, which is dual to the basis 
$\{H_i:i\in\mathbb{N}_n\}$ of $\mathfrak{h}_{\overline{0}}$. Then 
$\Phi=\{\varepsilon_i-\varepsilon_j:i,j\in\mathbb{N}_n,i\neq j\}$
is the {\em root system} of $\mathfrak{q}$ with the corresponding
Weyl group $W\cong \mathbf{S}_n$. We also have the standard set  
$\Phi^+=\{\varepsilon_i-\varepsilon_j:i,j\in\mathbb{N}_n,i<j\}$ 
of {\em positive} roots. The group $W$ acts on 
$\mathfrak{h}^*_{\overline{0}}$ in the usual way and via the
dot action $w\cdot\lambda=w(\lambda+\rho)-\rho$, where
$\rho$ is the half of the sum of all positive roots.

For $\alpha\in \Phi$ set
$\mathfrak{q}^{\alpha}=\{v\in \mathfrak{q}:[H,v]=\alpha(H)v
\text{ for all }H\in \mathfrak{h}_{\overline{0}}\}$. Then
$\mathfrak{q}^{\alpha}=\mathfrak{q}^{\alpha}_{\overline{0}}
\oplus\mathfrak{q}^{\alpha}_{\overline{1}}$, where both components
are one-dimensional, and we further have the decomposition 
\begin{displaymath}
\mathfrak{q}= \mathfrak{h}\oplus
\bigoplus_{\alpha\in \Phi}\mathfrak{q}^{\alpha}.
\end{displaymath}
This induces the natural triangular decomposition
$\mathfrak{q}=\mathfrak{n}_-\oplus \mathfrak{h}\oplus \mathfrak{n}_+$
with respect to our choice $\Phi^+$ of positive roots.

Elements in $\mathfrak{h}^*_{\overline{0}}$ are called {\em weights}
and are written $\lambda=(\lambda_1,\dots,\lambda_n)$ with respect to 
the basis $\{\varepsilon_i\}$. A weight $\lambda$ is called 
\begin{itemize}
\item {\em integral} provided that $\lambda_i\in\mathbb{Z}$
for any $i\in\mathbb{N}_n$;
\item {\em dominant} provided that $\lambda_i\in\mathbb{Z}$
for any $i\in\mathbb{N}_n$ and 
$\lambda_i-\lambda_{i-1}\in\mathbb{N}_0$  for any $i\in\mathbb{N}_{n-1}$;
\item {\em regular} provided that 
$\lambda_i-\lambda_j\in\mathbb{Z}\setminus\{0\}$ for any 
$i,j\in\mathbb{N}_n$, $i\neq j$;
\item {\em typical} provided that $\lambda_i+\lambda_j\neq 0$ for all
$i,j\in\mathbb{N}_n$, $i\neq j$;
\item {\em strongly typical} provided that $\lambda_i+\lambda_j\neq 0$ 
for all $i,j\in\mathbb{N}_n$.
\end{itemize}
For $\lambda,\mu\in \mathfrak{h}^*_{\overline{0}}$ we write 
$\lambda\leq \mu$ provided that $\mu-\lambda\in \mathbb{N}_0\Phi^+$.

The algebra $\mathfrak{q}_{\overline{0}}$ is identified with the 
Lie  algebra $\mathtt{g}=\mathfrak{gl}_n$ in the obvious way, 
and $\mathtt{g}$ inherits from $\mathfrak{q}$ the triangular decomposition 
$\mathtt{g}=\mathtt{n}_-\oplus \mathtt{h}\oplus\mathtt{n}_+$.
A $\mathfrak{q}$-supermodule $M$ is called a 
{\em weight} supermodule if it is a weight  (that is 
$\mathtt{h}$-diagonalizable) module with respect to $\mathtt{g}$.

We consider the category $\mathfrak{SM}$ of all 
$\mathfrak{q}$-supermodules, where all morphisms are homogeneous maps 
of degree zero. This is an abelian category with usual kernels and cokernels.
Let $\Pi:\mathfrak{SM}\to \mathfrak{SM}$ denote the autoequivalence,
which changes the parity. Let further $\mathfrak{M}$ denote the category of
all $\mathtt{g}$-modules, which is also abelian with usual kernels and
cokernels. Let $\mathrm{Res}^{\mathfrak{q}}_{\mathtt{g}}:\mathfrak{SM}\to
\mathfrak{M}$ denote the functor of restriction from $\mathfrak{q}$ 
to $\mathtt{g}$, which sends a $\mathfrak{q}$-supermodule $M=M_{\overline{0}}\oplus M_{\overline{1}}$
to the $\mathtt{g}$-module $M_{\overline{0}}$. We denote by
$\mathrm{Ind}^{\mathfrak{q}}_{\mathtt{g}}:\mathfrak{M}\to
\mathfrak{SM}$ the left adjoint of 
$\mathrm{Res}^{\mathfrak{q}}_{\mathtt{g}}$.

Let $\mathcal{O}^{\mathfrak{q}}$ and $\mathcal{O}^{\mathtt{g}}$ denote the 
BGG categories $\mathcal{O}$ for $\mathfrak{q}$ and $\mathtt{g}$, respectively
(see \cite{BGG}). These are full subcategories in the respective categories
of finitely generated (super)module, which consist of weight
(super)modules, which are locally $U(\mathfrak{n}_+)$- and 
$U(\mathtt{n}_+)$-finite, respectively.

Let $U(\mathfrak{q})$ and $U(\mathtt{g})$ denote the {\em universal enveloping
(super)algebra} of $\mathfrak{q}$ and $\mathtt{g}$, respectively.
Let, further, $Z(\mathfrak{q})$ and $Z(\mathtt{g})$ denote the 
{\em (super)center} of $U(\mathfrak{q})$ and $U(\mathtt{g})$, respectively. 
The action of the (super)center gives rise to the following 
{\em block decomposition} of $\mathcal{O}^{\mathfrak{q}}$ and 
$\mathcal{O}^{\mathtt{g}}$, indexed by {\em central characters}:
\begin{displaymath}
\mathcal{O}^{\mathfrak{q}}=\oplus_{\chi}
\mathcal{O}^{\mathfrak{q}}_{\chi},\quad\quad\quad
\mathcal{O}^{\mathtt{g}}=\oplus_{\chi}
\mathcal{O}^{\mathtt{g}}_{\chi}.
\end{displaymath}
For any $\chi$ we have the inclusion functor 
$\mathrm{incl}_{\chi}:\mathcal{O}_{\chi}\to \mathcal{O}$ and the
projection functor $\mathrm{proj}_{\chi}:\mathcal{O}\to \mathcal{O}_{\chi}$,
which are both left and right adjoint to each other.

Throughout the paper we fix a regular dominant strongly 
typical weight $\lambda$ (for $\mathfrak{q}$) and denote by 
$\hat{\lambda}$ its restriction to a $\mathtt{g}$-weight (note that
we formally have $\lambda=\hat{\lambda}$ as elements 
in $\mathfrak{h}^*_{\overline{0}}$, however, it is convenient to 
use different notation to specify the algebra we work with).
Let $\chi=\chi_{\lambda}$ and $\hat{\chi}$ be the central characters
for $\mathfrak{q}$ and $\mathtt{g}$, which correspond to 
$\lambda$ and $\hat{\lambda}$, respectively. We also denote by
$\mathfrak{m}_{\chi}$ the kernel of $\chi$ and by
$\mathtt{m}_{\chi}$ the kernel of $\hat{\chi}$. Then the block 
$\mathcal{O}^{\mathtt{g}}_{\chi}$ in indecomposable. If $n$ is odd,
the block $\mathcal{O}^{\mathfrak{q}}_{\chi}$ is indecomposable as well.
If $n$ is even, then 
\begin{displaymath}
\mathcal{O}^{\mathfrak{q}}_{\chi}\cong
\tilde{\mathcal{O}} ^{\mathfrak{q}}_{\chi}\oplus 
\Pi \tilde{\mathcal{O}} ^{\mathfrak{q}}_{\chi}
\end{displaymath}
for some indecomposable $\tilde{\mathcal{O}} ^{\mathfrak{q}}_{\chi}$.
To make our notation independent of the parity of $n$ we will just write
$\tilde{\mathcal{O}}^{\mathfrak{q}}_{\chi}$ to denote an indecomposable
direct summand of $\mathcal{O}^{\mathfrak{q}}_{\chi}$ for all $n$.
Abusing notation we will denote the inclusion and projection functors
between $\tilde{\mathcal{O}}^{\mathfrak{q}}_{\chi}$ and
$\mathcal{O}^{\mathfrak{q}}$ in the same way as above.

We have the following pair of functors:
\begin{displaymath}
\xymatrix{ 
\tilde{\mathcal{O}}^{\mathfrak{q}}_{\chi}
\ar@/^/[rrrrrr]^{\mathrm{G}:=\mathrm{proj}_{\hat{\chi}}\circ
\mathrm{Res}^{\mathfrak{q}}_{\mathtt{g}}\circ
\mathrm{incl}_{\chi}}
&&&&&&
\mathcal{O}^{\mathtt{g}}_{\hat{\chi}}
\ar@/^/[llllll]^{\mathrm{F}:=\mathrm{proj}_{\chi}
\circ\mathrm{Ind}^{\mathfrak{q}}_{\mathtt{g}}\circ
\mathrm{incl}_{\hat{\chi}}}
}
\end{displaymath}
From the definitions we have that $(\mathrm{F},\mathrm{G})$ is an 
adjoint pair of functors. The main result of this paper is the 
following theorem:

\begin{theorem}\label{thmmain}
There is a direct summand $\mathrm{F}_1$ of $\mathrm{F}$ and a direct 
summand $\mathrm{G}_1$ of $\mathrm{G}$ such that the functors 
$\mathrm{F}_1$ and $\mathrm{G}_1$ are mutually inverse equivalences 
of categories.
\end{theorem}
 
Before we proceed it is necessary to say why the original functors
$\mathrm{F}$ and $\mathrm{G}$ are not equivalences of categories. Both 
$\tilde{\mathcal{O}}^{\mathfrak{q}}_{\chi}$ 
and $\mathcal{O}^{\mathtt{g}}_{\hat{\chi}}$
are equivalent to categories of modules over finite-dimensional 
quasi-hereditary algebras (see \cite{BGG} for 
$\mathcal{O}^{\mathtt{g}}_{\hat{\chi}}$ and \cite{Fr} for
$\tilde{\mathcal{O}}^{\mathfrak{q}}_{\chi}$). Moreover, simple objects
in both categories are naturally indexed by elements from $W$. 
Quasi-hereditary structure on both categories comes with the collection
of standard modules. 

In  $\mathcal{O}^{\mathtt{g}}_{\hat{\chi}}$ standard modules are the 
usual Verma modules $M(\mu)$, $\mu\in W\cdot \lambda$ (observe that 
here we have the dot action of $W$).
In $\tilde{\mathcal{O}}^{\mathfrak{q}}_{\chi}$ standard modules are 
{\em Verma supermodules}. They are 
constructed as follows: For $\mu\in\mathfrak{h}^*_{\overline{0}}$
let $V(\mu)$ be a simple $\mathfrak{h}$-supermodule of
weight $\mu$. The supermodule $V(\mu)$ is unique if $n$ is odd
and satisfies $\Pi(V(\mu))\cong V(\mu)$. If $n$ is even then there
are exactly two simple $\mathfrak{h}_{\overline{0}}$-supermodules of
weight $\mu$, namely $\Pi(V(\mu))$ and $V(\mu)$ (we denote by 
$V(\mu)$ the supermodule, which will give rise to the Verma supermodule
in $\tilde{\mathcal{O}}^{\mathfrak{q}}_{\chi}$). 
We have  $V(\mu)=V(\mu)_{\overline{0}}\oplus V(\mu)_{\overline{1}}$ and 
\begin{displaymath}
\dim_{\Bbbk}(V(\mu)_{\overline{0}})=\dim_{\Bbbk}(V(\mu)_{\overline{1}}) 
=2^{\lfloor(n-1)/2\rfloor}=:\mathbf{k}.
\end{displaymath}
Letting $\mathfrak{n}_+$ act trivially on $V(\mu)$ and inducing the 
obtained module up to $U(\mathfrak{q})$ we obtain the {\em Verma supermodule}
$\Delta(V(\mu))$ (note that these  supermodules where called 
{\em Weyl modules} in \cite{Go2}). The weight $\mu$ is a highest weight of
$\Delta(V(\mu))$ and has both even and odd dimension $\mathbf{k}$. 
The standard modules in $\tilde{\mathcal{O}}^{\mathfrak{q}}_{\chi}$
are $\Delta(V(\mu))$, $\mu\in W\lambda$ (observe that here we have the
usual action of $W$).

\begin{proposition}\label{prop2}
For every $w\in W$ we have
\begin{gather*}
\mathrm{G}\Delta(V(w\lambda))\cong 
\underbrace{M(w\cdot\lambda)\oplus M(w\cdot\lambda)\oplus \dots\oplus 
M(w\cdot\lambda)}_{\mathbf{k}\text{ summands }}=:
\mathbf{k}M(w\cdot\lambda),\\
\mathrm{F}M(w\cdot\lambda)\cong 
\underbrace{\Delta(V(w\lambda))\oplus  
\dots\oplus  \Delta(V(w\lambda))}_{\mathbf{k}\text{ summands }}=:
\mathbf{k}\Delta(V(w\lambda)).
\end{gather*}
\end{proposition}

\begin{proof}
We will need the following combinatorial lemma:

\begin{lemma}\label{lem3}
Let $P$ denote the multiset $\{\sum_{\alpha\in I}\alpha:I\subset \Phi^+\}$.
\begin{enumerate}[(i)]
\item\label{lem3-1} 
For any $w\in W$ we have
\begin{displaymath}
(\{w\cdot \lambda\}+P)\cap W\lambda=\{w\lambda\},\quad
(\{w\lambda\}-P)\cap W\cdot \lambda=\{w\cdot \lambda\}.
\end{displaymath}
\item\label{lem3-2}
For any $w\in W$ there is a unique element $x\in P$ such that we have 
$w\cdot\lambda+x=w\lambda$.
\end{enumerate}
\end{lemma}

\begin{proof}
The second equality of \eqref{lem3-1} follows from the first one, 
so we will prove only the first equality. Fix $w\in W$. Then we have 
\begin{displaymath}
\rho-w\rho=\frac{1}{2}\sum_{\alpha\in\Phi^+}\alpha-
\frac{1}{2}\sum_{\beta\in w\Phi^+}\beta=
\sum_{\alpha\in\Phi^+, w\alpha\not\in \Phi^+}\alpha,
\end{displaymath}
which is an element of $P$. This yields $w\lambda-w\cdot \lambda\in P$ 
and hence  $w\lambda\in \{w\cdot \lambda\}+P$.

On the other hand, as $\lambda$ is dominant integral,
for any $w'\in W$ we have 
\begin{equation}\label{eq3}
w\lambda-w'\lambda\in \mathbb{N}_0w\Phi^+.
\end{equation}
Set $P_w=\{\sum_{\alpha\in I}\alpha:I\subset w\Phi^+\}\subset 
\mathbb{N}_0w\Phi^+$. For any 
$\alpha\in\Phi^+$ we either have $\alpha\in w\Phi^+$ or
$-\alpha\in w\Phi^+$. This implies that $\rho+P_w=w\rho+P$ 
(bijection of multisets) and hence
\begin{displaymath}
w\cdot\lambda+P=w\lambda+w\rho-\rho+P=w\lambda+P_w. 
\end{displaymath}
Thus any element $v\in w\cdot\lambda+P$ satisfies 
\begin{equation}\label{eq4}
v- w\lambda\in \mathbb{N}_0w\Phi^+.
\end{equation}
The claim \eqref{lem3-1} follows comparing \eqref{eq3} and \eqref{eq4}.
The claim \eqref{lem3-2} follows from the observation that the only
common element of \eqref{eq3} and \eqref{eq4} is obtained when 
$w=w'$ and $v=0$ (that is such $v$ is unique).
\end{proof}

Fix $w\in W$. Consider the module
$N=U(\mathfrak{g})\otimes_{U(\mathtt{g})}M(w\cdot \lambda)$. 
By the PBW Theorem for Lie superalgebras (\cite{Ro}), the algebra 
$U(\mathfrak{g})$ is a free $U(\mathtt{g})$-module with basis
$\bigwedge \mathfrak{g}_{\overline{1}}$. In particular, it
follows that the module $N$ is free over $U(\mathfrak{n}_-)$
of finite rank. Hence $N$ has a filtration whose subquotients 
are Verma supermodules. As every Verma supermodule is
free over $U(\mathfrak{n}_-)$ by definition, we obtain that the
highest weight elements of subquotients in any Verma filtration of 
$N$ are images of the elements  from the space 
$\bigwedge (\mathfrak{n}_{+})_{\overline{1}}  \otimes v$, where $v$ is 
a highest weight vector of $M(w\cdot \lambda)$. Hence the corresponding
highest weights  belong to $w\cdot\lambda+P$. By Lemma~\ref{lem3}\eqref{lem3-1},
the only weight from $W\lambda$, which intersects $w\cdot\lambda+P$
is $w\lambda$. This yields that $\mathrm{F}M(w\cdot\lambda)$ is
isomorphic to the direct sum of several copies of $\Delta(V(w\lambda))$,
say $k$ copies.

Similarly, the restriction of $\Delta(V(w\lambda))$ to $U(\mathtt{g})$
is a $U(\mathtt{n}_-)$-free module of finite rank, and hence has 
a Verma filtration. The highest weight elements of subquotients of 
this Verma filtration are images of elements from
$\bigwedge (\mathfrak{n}_{-})_{\overline{1}}  \otimes v$, where 
$v$ is a highest weight vector  of $\Delta(V(w\lambda))$. 
Hence the corresponding
highest weights have form $w\lambda-P$. By Lemma~\ref{lem3}\eqref{lem3-1},
the only weight from $W\cdot \lambda$, which intersects $w\lambda-P$
is $w\cdot \lambda$. This yields that $\mathrm{G}\Delta(V(w\lambda))$ is
isomorphic to the direct sum of several copies of $M(w\cdot \lambda)$,
say $m$ copies.

By adjunction we have 
\begin{displaymath}
\mathrm{Hom}_{U(\mathfrak{g})} 
(\mathrm{F}M(w\cdot\lambda),\Delta(V(w\lambda)))=
\mathrm{Hom}_{U(\mathtt{g})}
(M(w\cdot\lambda),\mathrm{G}\Delta(V(w\lambda))),
\end{displaymath}
which yields $k=m$ as Verma (super)modules have trivial endomorphism ring.

Finally, the multiplicity $m$ equals the even multiplicity of the weight
$w\cdot \lambda$ in the space 
$\bigwedge (\mathfrak{n}_{-})_{\overline{1}}  \otimes V(w\lambda)$.
By Lemma~\ref{lem3}\eqref{lem3-2}, the multiplicity of the
weight $w\cdot \lambda-w\lambda$ in $\bigwedge
(\mathfrak{n}_{-})_{\overline{1}}$ equals $1$.
Since $\dim_{\Bbbk}(V(w\lambda)_{\overline{0}})=
\dim_{\Bbbk}(V(w\lambda)_{\overline{1}})=\mathbf{k}$, it follows that
$m=\mathbf{k}$ and the proof is  complete.
\end{proof}

From Proposition~\ref{prop2} it follows that in order to prove 
Theorem~\ref{thmmain} we have to decomposethe functors  $\mathrm{F}$
and $\mathrm{G}$ into a direct sum of $\mathbf{k}$ copies of isomorphic
functors. For this we will need the technique of Harish-Chandra bimodules.

\section{Harish-Chandra 
$U(\mathfrak{q})-U(\mathtt{g})$-bimodules}\label{s3} 

This section is inspired by and based on \cite[3.1.2]{Go1}. 
Let $M,N$ be $U(\mathtt{g})$-modules. Then the space 
$\mathrm{Hom}_{\mathbb{C}}(M,N)$ has the natural structure of
a $U(\mathtt{g})$-bimodule and contains the subbimodule 
$\mathcal{L}(M,N)$, consisting of all elements, the adjoint action of 
$\mathtt{g}$ on which is locally finite (see \cite[Kapitel~6]{Ja}). 
Similarly one defines $\mathcal{L}(M,N)$ in the case 
$M$ and $N$ are $U(\mathfrak{q})$-modules and in the case 
$M$ is a $U(\mathtt{g})$-module module and $N$ is a 
$U(\mathfrak{q})$-module. As $U(\mathfrak{q})$ is a finite extension
of $U(\mathtt{g})$ (\cite{Ro}), in all cases we can impose the 
condition that the adjoint action of the Lie algebra
$\mathtt{g}$ on elements from $\mathcal{L}(M,N)$ is locally finite.

Fix $r\in\mathbb{N}$. Let $\mathtt{m}$ denote the maximal ideal of
$U(\mathtt{h})$, generated by the elements $H-\lambda(H)$,
$H\in \mathtt{h}$. Consider the $U(\mathtt{h})$-module
$U(\mathtt{h})/\mathtt{m}^r$.  Let $\mathtt{n}_+$ act on 
$U(\mathtt{h})/\mathtt{m}^r$ trivially and construct the
{\em thick Verma module} $M^r(\lambda)$ as follows (see \cite{So}):
\begin{displaymath}
M^r(\lambda)=U(\mathtt{g}) \bigotimes_{U(\mathtt{h}\oplus \mathtt{n}_+)}
U(\mathtt{h})/\mathtt{m}^r.
\end{displaymath}
Since $\lambda$ is regular, by \cite{So} we have 
\begin{equation}\label{eq5}
\mathrm{Ann}_{U(\mathtt{g})}(M^r(\lambda))=
U(\mathtt{g})\mathtt{m}_{\chi}^r,\quad
\mathcal{L}(M^r(\lambda),M^r(\lambda))\cong
U(\mathtt{g})/U(\mathtt{g})\mathtt{m}_{\chi}^r.
\end{equation}
 
Our main result in this section is the following statement, which
generalizes \cite[3.1.2]{Go1} to our setup:

\begin{proposition}\label{prop5}
There is an isomorphism of $U(\mathfrak{q})-U(\mathtt{g})$-bimodules
as follows:
\begin{displaymath}
U(\mathfrak{q})/U(\mathfrak{q})\mathtt{m}_{\chi}^r\cong
\mathcal{L}(M^r(\lambda),
\mathrm{Ind}_{\mathtt{g}}^{\mathfrak{q}}M^r(\lambda)).
\end{displaymath}
\end{proposition}

\begin{proof}
Consider the homomorphism $\varphi$  of
$U(\mathfrak{q})-U(\mathtt{g})$-bimodules, defined as follows: 
$\varphi: U(\mathfrak{q})\rightarrow \mathcal{L}(M^r(\lambda),
\mathrm{Ind}_{\mathtt{g}}^{\mathfrak{q}}M^r(\lambda))$, where
$\varphi(u)(m)=u\otimes m$. By \eqref{eq5} we have
$\mathrm{Ann}_{U(\mathtt{g})}(M^r(\lambda))=
U(\mathtt{g})\mathtt{m}_{\chi}^r$ and hence the map
$\varphi$ induces a $U(\mathfrak{q})-U(\mathtt{g})$-bimodule homomorphism 
\begin{displaymath}
\overline{\varphi}:U(\mathfrak{q})/U(\mathfrak{q})\mathtt{m}_{\chi}^r
\rightarrow \mathcal{L}(M^r(\lambda),
\mathrm{Ind}_{\mathtt{g}}^{\mathfrak{q}}M^r(\lambda)).
\end{displaymath}
Since $U(\mathfrak{q})$ is free over $U(\mathtt{g})$, we conclude that
$\overline{\varphi}$ is injective.
Let us prove that $\overline{\varphi}$ is surjective. 

By the PBW theorem we have $U(\mathfrak{q})\cong 
\bigwedge \mathfrak{q}_{\overline{1}}\otimes U(\mathtt{g})$.
By Kostant separation theorem (see \cite{Ko}), there is a submodule
$H$ of the adjoint $\mathtt{g}$-module $U(\mathtt{g})$ such that 
$U(\mathtt{g})\cong H\otimes Z(\mathtt{g})$. This gives us the
following isomorphism of adjoint $\mathtt{g}$-modules:
\begin{displaymath}
U(\mathfrak{q})/U(\mathfrak{q})\mathtt{m}_{\chi}^r\cong
\bigwedge \mathfrak{q}_{\overline{1}}\otimes H\otimes 
Z(\mathtt{g})/\mathtt{m}_{\chi}^r.
\end{displaymath}

On the other hand, since $\bigwedge \mathfrak{q}_{\overline{1}}$ 
is finite-dimensional, we also have the
following isomorphism of adjoint $\mathtt{g}$-modules:
\begin{eqnarray*}
\mathcal{L}(M^r(\lambda),
\mathrm{Ind}_{\mathtt{g}}^{\mathfrak{q}}M^r(\lambda))
&\overset{\text{PBW}}{=}&
\mathcal{L}(M^r(\lambda),
\bigwedge \mathfrak{q}_{\overline{1}} \otimes M^r(\lambda))\\ 
\text{(by \cite[6.8]{Ja})} &=&
\bigwedge \mathfrak{q}_{\overline{1}}\otimes \mathcal{L}(M^r(\lambda),
M^r(\lambda))\\
\text{(by \eqref{eq5})}
&=&\bigwedge \mathfrak{q}_{\overline{1}}\otimes 
U(\mathtt{g})/U(\mathtt{g})\mathtt{m}_{\chi}^r\\
&=&\bigwedge \mathfrak{q}_{\overline{1}}\otimes  H\otimes Z(\mathtt{g})/\mathtt{m}_{\chi}^r .
\end{eqnarray*}
The claim of the proposition follows.
\end{proof}

\section{Proof of the main result}\label{s4} 

We will need the following lemma:

\begin{lemma}\label{lem6}
There exists $r\in\mathbb{N}$ such that $\mathtt{m}_{\chi}^r M=0$
for any $M\in \mathcal{O}^{\mathtt{g}}_{\hat{\chi}}$.
\end{lemma}

\begin{proof}
The category $\mathcal{O}^{\mathtt{g}}_{\hat{\chi}}$ is equivalent to
the category of modules over some finite-dimensional algebra
(\cite{BGG}). Hence it has a projective generator $Q$ of finite length,
say $r$, and every object in $\mathcal{O}^{\mathtt{g}}_{\hat{\chi}}$
is a quotient of a direct sum of some copies of $Q$. As
$\mathtt{m}_{\chi}$ annihilates all $M(w\cdot\lambda)$, we have
that $\mathtt{m}_{\chi}$ annihilates all simple objects
and hence  that $\mathtt{m}_{\chi}^r$ annihilates $Q$. The claim follows.
\end{proof}

Now we have to define thick Verma supermodule $\Delta(V^r(\lambda))$. 
Let $V^r(\lambda)$ denote the indecomposable projective cover of 
the simple object $V(\lambda)$ in the category $\mathfrak{F}_{\lambda}^r$
of all finite-dimensional $\mathfrak{h}$-supermodules, annihilated by
$\mathtt{m}^r$. Let $\mathfrak{n}_+$ act on 
$V^r(\lambda)$ trivially and define the
{\em thick Verma supermodule} $\Delta(V^r(\lambda))$ as follows:
\begin{displaymath}
\Delta(V^r(\lambda))=U(\mathfrak{q}) \bigotimes_{U(\mathfrak{h}\oplus 
\mathfrak{n}_+)} V^r(\lambda).
\end{displaymath}
Let $\hat{\mathfrak{F}}_{\lambda}^r$ denote the category of all
finite-dimensional $\mathtt{h}$-modules, annihilated by
$\mathtt{m}^r$. Then $U(\mathtt{h})/\mathtt{m}^r$
is the unique (up to isomorphism) indecomposable projective module
in  $\hat{\mathfrak{F}}_{\lambda}^r$.

\begin{lemma}\label{lem7}
We have $\mathrm{proj}_{\chi}\circ
\mathrm{Ind}_{\mathtt{g}}^{\mathfrak{q}} M^r(\lambda)\cong
\mathbf{k} \Delta(V^r(\lambda))$,
\end{lemma}

\begin{proof}
By adjunction it is enough to show that the kernel of 
$\mathtt{m}_{\chi}^r$ on the $U(\mathtt{g})$-module 
$\Delta(V^r(\lambda))_{\overline{0}}$ is isomorphic to 
$\mathbf{k} M^r(\lambda)$. 

\begin{lemma}\label{lem8}
We have 
$\mathrm{Res}^{\mathfrak{h}}_{\mathtt{h}}
V^r(\lambda)_{\overline{0}}\cong 
\mathbf{k}\, U(\mathtt{h})/\mathtt{m}^r$.
\end{lemma}

\begin{proof}
The restriction functor is left adjoint to the coinduction functor
$\mathrm{Coind}^{\mathfrak{h}}_{\mathtt{h}}\cong
\Pi^n\mathrm{Ind}^{\mathfrak{h}}_{\mathtt{h}}$,
the latter being exact (\cite[Proposition~22]{Fr}). Hence
$\mathrm{Res}^{\mathfrak{h}}_{\mathtt{h}}$ maps $V^r(\lambda)$ 
to a projective module in $\hat{\mathfrak{F}}_{\lambda}^r$, that is 
$\mathrm{Res}^{\mathfrak{h}}_{\mathtt{h}}
V^r(\lambda)_{\overline{0}}\cong k\, U(\mathtt{h})/\mathtt{m}^r$.
Let $\Bbbk_{\lambda}$ denote the simple object in 
$\hat{\mathfrak{F}}_{\lambda}^r$. We have 
\begin{eqnarray*}
k&=&\dim\mathrm{Hom}_{\hat{\mathfrak{F}}_{\lambda}^r} 
(\mathrm{Res}^{\mathfrak{h}}_{\mathtt{h}}
V^r(\lambda)_{\overline{0}},\Bbbk_{\lambda} )\\
\text{(by adjunction)}&=&\dim\mathrm{Hom}_{\mathfrak{F}_{\lambda}^r} 
(V^r(\lambda),
\mathrm{Coind}^{\mathfrak{h}}_{\mathtt{h}}\Bbbk_{\lambda} )\\
\text{(by \cite[Proposition~22]{Fr})}
&=&\dim\mathrm{Hom}_{\mathfrak{F}_{\lambda}^r} 
(V^r(\lambda),
\Pi^n\mathrm{Ind}^{\mathfrak{h}}_{\mathtt{h}}\Bbbk_{\lambda} )\\
\text{(by projectivity)}&=&
\mathrm{length}(\Pi^n\mathrm{Ind}^{\mathfrak{h}}_{\mathtt{h}}\Bbbk_{\lambda})\\
\text{(by PBW)}&=&\mathbf{k}.
\end{eqnarray*}
\end{proof}

The claim of Lemma~\ref{lem7} follows from Lemma~\ref{lem8} by the
same arguments as in the proof of Proposition~\ref{prop2}.
\end{proof}

\begin{proof}[Proof of Theorem~\ref{thmmain}.]
Choose $r$ as given by Lemma~\ref{lem6}. Then the functor 
$\mathrm{F}$ is a direct summand (defined by the projection on 
$\tilde{\mathcal{O}}^{\mathfrak{q}}_{\chi}$) of the functor 
\begin{displaymath}
U(\mathfrak{q})/U(\mathfrak{q})\mathtt{m}_{\chi}^r
\bigotimes_{U(\mathtt{g})}{}_-.
\end{displaymath}
By Proposition~\ref{prop5}, the latter functor is isomorphic to the 
functor
\begin{displaymath}
\mathcal{L}(M^r(\lambda),
\mathrm{Ind}_{\mathtt{g}}^{\mathfrak{q}}M^r(\lambda))
\bigotimes_{U(\mathtt{g})}{}_-. 
\end{displaymath}
 
By Lemma~\ref{lem7} we have that the module 
$\mathrm{proj}_{\chi}\circ
\mathrm{Ind}_{\mathtt{g}}^{\mathfrak{q}}M^r(\lambda)$
decomposes into a direct sum of $\mathbf{k}$ copies of 
$\Delta(V^r(\lambda))$. Hence the additivity of the functor 
$\mathcal{L}(X,{}_-)$ implies that the functor $\mathrm{F}$ decomposes
into a direct sum of $\mathbf{k}$ copies of some functor 
$\mathrm{F}_1$. By adjunction, the adjoint $\mathrm{G}$ decomposes
into a direct sum of $\mathbf{k}$ copies of some functor 
$\mathrm{G}_1$ such that $(\mathrm{F}_1,\mathrm{G}_1)$ forms an
adjoint pair.

From Proposition~\ref{prop2} we have 
\begin{displaymath}
\mathrm{F}_1 M(w\cdot\lambda)\cong \Delta(V(w\lambda)),\quad
\mathrm{G}_1 \Delta(V(w\lambda))\cong M(w\cdot\lambda)
\end{displaymath}
for all $w\in W$. As all Verma (super)modules are not annihilated 
by $\mathrm{F}_1$ and $\mathrm{G}_1$, respectively, it follows that 
the adjunction morphisms 
$\mathrm{Id}_{\mathcal{O}^{\mathtt{g}}_{\hat{\chi}}}
\rightarrow \mathrm{G}_1\mathrm{F}_1$ and
$\mathrm{F}_1\mathrm{G}_1\rightarrow
\mathrm{Id}_{\tilde{\mathcal{O}}^{\mathfrak{q}}_{\chi}}$ are 
nonzero. As the endomorphism ring of a Verma (super)module is trivial,
it follows that these adjunction morphisms are isomorphisms.

As any simple object in both $\mathcal{O}^{\mathtt{g}}_{\hat{\chi}}$ and 
$\tilde{\mathcal{O}}^{\mathfrak{q}}_{\chi}$ 
is a unique quotient of some Verma (super)module,
we conclude that the adjunction morphisms are 
isomorphisms, when evaluated on all simple objects. As any object in 
$\mathcal{O}^{\mathtt{g}}_{\hat{\chi}}$ and 
$\tilde{\mathcal{O}}^{\mathfrak{q}}_{\chi}$ has finite length,
a standard induction on the length implies that the 
adjunction morphisms are isomorphisms of functors. This completes
the proof.
\end{proof}

\section{Concluding remarks}\label{s5} 

\begin{remark}\label{rem101}
{\rm
The main result transfers in a straightforward way to all other 
$Q$-type Lie superalgebras (as in \cite{Go2}).
}
\end{remark}

\begin{remark}\label{rem102}
{\rm
The main result easily generalizes to the nonintegral 
(regular strongly typical) case. We decided to work with the 
integral case to avoid complicated notation, which would heavily
decrease the readability of the paper.
}
\end{remark}

\begin{remark}\label{rem103}
{\rm
To extend the main result to singular weights one has to develop the
theory of Harish-Chandra bimodules for superalgebras in the similar
way as done for Lie algebras in \cite{BG}.
}
\end{remark}

\begin{remark}\label{rem104}
{\rm
Blocks which are typical but not strongly typical are described by 
properly stratified rather than by quasi-hereditary algebras
(see \cite{Fr}). Hence the results of this paper do not extend to
such blocks. In \cite[3.9]{Fr} it is conjectures that such blocks are
tensor products of strongly typical blocks with 
$\Bbbk[x]/(x^2)$-mod.
}
\end{remark}

\vspace{1cm}

\noindent
Department of Mathematics, Uppsala University, Box 480, SE-75106,
Uppsala, SWEDEN 
\vspace{2mm}

\noindent
{\tt frisken\symbol{64}math.uu.se}\\
{\tt mazor\symbol{64}math.uu.se}

\end{document}